\documentclass[a4paper]{amsart}
\usepackage{amsmath,amscd,amssymb}
\numberwithin{equation}{section}
\usepackage[varg]{txfonts}
\usepackage[all]{xy}
\makeatletter\@ifundefined{coloneqq}{}\makeatother
\newcommand{\ndeg}{\mathit{{D}}}
\newcommand{\hdef}{\mathit{{H}}}
\newcommand{\Z}{\mathbb{Z}}
\newcommand{\R}{\mathbb{R}}
\newcommand{\C}{\mathbb{C}}
\newcommand{\HH}{\mathbb{H}}

\newcommand{\ind}{\mathop{\mathrm{ind}}\nolimits}
\newcommand{\rank}{\mathop{\mathrm{rk}}\nolimits}

\newcommand{\imm}{\mathop{\mathrm{Cob}}\nolimits}
\newcommand{\Imm}{\mathop{\mathrm{Imm}}\nolimits}

\newcommand{\SO}{\mathop{\mathit{SO}}\nolimits}
\newcommand{\co}{\mathpunct{\colon}}
\newcommand{\ST}{\mathpunct{;}\,\,}
\newcommand{\pa}{\partial}
\renewcommand{\Re}{\operatorname{Re}}
\newtheorem{thm}{Theorem}[section]

\newtheorem{lem}[thm]{Lemma}
\newtheorem{prop}[thm]{Proposition}
\theoremstyle{definition}
\newtheorem{defn}[thm]{Definition}
\newtheorem{rem}[thm]{Remark}
\date{}
\title[Singular Seifert surfaces and Smale invariants]
{Singular Seifert surfaces and Smale invariants for a family of 3-sphere immersions}
\author{Tobias Ekholm and Masamichi Takase}
\address{Tobias Ekholm: 
Department of mathematics, Uppsala University, Box 480, 751 06 Uppsala, Sweden.}
\email{tobias@math.uu.se}
\address{Masamichi Takase: 
Department of Mathematical Sciences, Faculty of Science, Shinshu University, Matsumoto, 390-8621 Japan.}
\email{takase@math.shinshu-u.ac.jp}
\thanks{TE acknowledges support from the G{\"o}ran Gustafsson Foundation for Research in Natural Sciences and Medicine. MT is partially supported by the Grant-in-Aid for Scientific Research (C), JSPS, Japan}
\begin{document}\sloppy
\maketitle

\begin{abstract}
A self-transverse immersion of the $2$-sphere into $4$-space with algebraic number of self intersection points equal to $-n$ induces an immersion of the circle bundle over the $2$-sphere of Euler class $2n$ into $4$-space. Precomposing these circle bundle immersions with their universal covering maps, we get for $n>0$ immersions $g_{n}$ of the $3$-sphere into $4$-space. In this note, we compute the Smale invariants of $g_n$. The computation is carried out by (partially) resolving the singularities of the natural singular map of the punctured complex projective plane which extends $g_n$.

As an application, we determine the classes represented by $g_n$ in the cobordism group of immersions which is naturally identified with the stable $3$-stem. It follows in particular that $g_n$ represents a generator of the stable $3$-stem if and only if $n$ is divisible by $3$.
\end{abstract}

\section{Introduction}\label{sec:intr}
Consider a self-transverse immersion $f\colon S^{2}\to\R^{4}$ of the $2$-sphere $S^2$ into $4$-space $\R^4$. Let $p\in\R^{4}$ be a self intersection point of $f$ and assume that $S^{2}$ and $\R^{4}$ are oriented. Ordering the two local sheets of $f(S^{2})$ intersecting at $p$ we get an intersection number at $p$. Since the codimension is even this intersection number is independent of the ordering. We say that the sum of the intersection numbers over all the double points of $f$ is its {\em algebraic number of double points}.

Let $f\colon S^{2}\to\R^{4}$ be an immersion with algebraic number of double points equal to $-n$. Then the normal bundle of $f$ is the oriented $2$-plane bundle $\xi_{2n}$ of Euler number $2n$, see \cite{whitney}, and the orientations of $S^{2}$ and $\R^{4}$ induce an orientation on $\xi_{2n}$. The unit sphere bundle of $\xi_{2n}$ is the lens space $L(2n,1)$ which we orient as the boundary of the unit disk bundle and it follows by micro-extension that $f$ induces an immersion $f_{2n}\colon L(2n,1)\to\R^{4}$. Consider the universal ($2n$-fold) covering $\pi_{2n}\colon S^{3}\to L(2n,1)$, where we orient $S^{3}$ so that $\pi_{2n}$ is orientation preserving. Precomposing we get an immersion
\begin{equation}\label{e:G_n}
g_n=f_{2n}\circ\pi_{2n}\colon S^{3}\to\R^{4}.
\end{equation}

Let $\Imm[S^m,\R^N]$ denote the group of regular homotopy classes of immersions $S^m\to\R^N$, where the group operation is induced by connected sum of immersions, see \cite{kervaire}. The Smale-Hirsch h-principle implies that the group $\Imm[S^m,\R^N]$ is isomorphic to the $m^{\rm th}$ homotopy group $\pi_m(V_{N,m})$ of the Stiefel manifold $V_{N,m}$ of $m$ frames in $\R^{N}$. The isomorphism is given by the Smale invariant
\[
\Omega\colon\Imm[S^m,\R^N]\to\pi_{m}(V_{N,m}),
\]
see \cite{smale}. In particular, in the case studied in this paper:
\[
\Imm[S^{3},\R^{4}]=\pi_3(V_{4,3})\approx\pi_3(\SO_4)\approx\Z\oplus\Z,
\]
see \eqref{eq:pi_3(SO_4)} for a description of the last isomorphism.

Let $\imm(m,N)$ denote the cobordism group of immersions of closed oriented $m$-manifolds into $\R^{N}$, where the group operation is induced by disjoint union. The cobordism group $\imm(3,4)$ is isomorphic to the stable $3$-stem
\[
\imm(3,4)\approx \pi^{S}_3\approx\Z_{24},
\]
see \cite{rohlin,wells}.

\begin{thm}\label{t:main}
The Smale invariant of the immersion $g_{n}\colon S^{3}\to\R^{4}$, $n>0$, satisfies
\[
\Omega(g_n)=\bigl(4n-1,(n-1)^{2}\bigr)\in\Z\oplus\Z.
\]
It follows in particular that $g_n$ represents the element
\[
(2n^{2}+1)\operatorname{mod}{24}\,\in\Z_{24}\approx\pi_{3}^{S},
\]
and hence generates the stable $3$-stem if and only if $n$ is a multiple of $3$.
\end{thm}
Theorem \ref{t:main} is proved in Section \ref{sect:main}, for other constructions of generators of the stable $3$-stem, see \cite{carter,hughes}. The proof uses singular Seifert surfaces, which were introduced in \cite{e-s0} and have since proved to be quite useful in the study of embeddings, immersions, and other maps. See for example \cite{andersson,e-s0,e-s1,juhasz,sst,takase}. In fact, one of the main technical points of this paper is a concrete construction of singular Seifert surfaces for the immersions $g_n$. The construction is based on perturbations derived from an unfolding of a certain complex map germ, regarded as a real map germ, and might be of interest from a more general viewpoint since similar constructions of stable or other singular maps with desired properties can be used in many other settings.

In Section \ref{sect:remarks} we give a brief discussion of relations to other results. We first discuss the immersion $g_1\co S^3\to\R^4$, studied
by Milnor \cite[(11), \S{IV}]{milnor} and since then
in several papers (e.g.~\cite{l-s,ekholm}).
In this case, Theorem \ref{t:main} recovers the result \cite[Proposition~8.4.1]{ekholm} which shows that $g_1$ represents a generator in $\Z_{8}\subset\Z_{24}$ as well as the fact that the triple point invariant $\lambda(g_1)\in\Z_3$ (defined in \cite[\S6.2]{ekholm}) of $g_1$ vanishes. See Remark \ref{rmk:linking} for an alternative proof of the latter. Second, we show that the immersion $g_2$ coincides, up to regular homotopy and choice of orientation on $S^{3}$, with Melikhov's example
of an immersion with non-trivial stable Hopf invariant \cite[Example~4]{melikhov}.

\subsection*{Notational conventions}
We work in the smooth category throughout;
all manifolds and immersions are assumed to be differentiable of class $C^\infty$.
Furthermore, all spheres and Euclidean spaces are assumed to be oriented and we orient the boundary of an oriented manifold by the ``outward normal first'' convention, as in e.g.~\cite{hughes,k-m}.

\section{Smale invariants in terms of singular Seifert surfaces}
In this section we first establish notation for the Smale invariant of immersions $S^{3}\to\R^{4}$ and review some of its properties. Then we present a formula for one of the components of the Smale invariant in terms of singularities of a Seifert surface.

\subsection{The regular homotopy class}\label{subsect:regular}
An immersion $f\co S^3\to\R^4$ comes equipped with a natural stable framing via the bundle isomorphism $\epsilon^1\oplus TS^3\cong f^*T\R^4$, where $\epsilon^1$ denotes the trivial line bundle. The homotopy class of the stable framing
is completely characterized by two integers: the degree $\ndeg(f)$ and the Hirzebruch defect $\hdef(f)$. We refer to \cite{k-m} for the detailed definitions and mention here only that $\ndeg(f)$ equals the normal degree of $f$, see \cite{milnor}.

The Smale-Hirsch h-principle implies that the map which associates to an immersion $S^{3}\to\R^{4}$ its natural stable framing induces a (weak) homotopy equivalence between the space of immersions and the space of stable framings. In particular, on $\pi_0$ (i.e.~on the level of path components) the Smale invariant, see Section \ref{sec:intr},
\[
\Omega(f)\in\pi_3(\SO_4)\approx\Z\oplus\Z,
\]
which determines the regular homotopy class of
an immersion $f\colon S^{3}\to\R^{4}$, is determined by $\ndeg(f)$ and $\hdef(f)$. In order to describe this in more detail we describe our conventions for $\Omega$ and $\pi_{3}(\SO_4)$.

We first describe an explicit isomorphism $\pi_3(\SO_4)\to\Z\oplus\Z$ following \cite{steenrod}: let $e_1,e_2,e_3,e_4$ denote the standard basis vectors in $\R^{4}$ and think of $\SO_{4}$ as the set of orthonormal transformations $R\colon\R^{4}\to\R^{4}$. Consider the fibration $p\colon\SO_4\to S^{3}$, $p(R)=R(e_1)$ with fiber $\SO_{3}$. Think of $S^{3}$ as the set of unit quaternions in $\HH\approx\R^{4}$ via the identification of the vectors $e_1,e_2,e_3,e_4\in\R^{4}$ with $1,i,j,k\in\HH$. For $x,y\in\R^{4}$, let $x\cdot y\in\R^{4}$ denote their quaternionic product. The map $\sigma\colon S^{3}\to\SO_4$, $\sigma(x)y=x\cdot y$ is a section of $p\colon\SO_4\to S^{3}$. Hence $\SO_4=S^{3}\times\SO_3$. Let $\rho\colon S^3\to\SO_4$ be the map given by $\rho(x)y=x\cdot y\cdot x^{-1}$. Then $p\circ\rho$ is the constant map to $e_1$ and $\rho$ is a double cover of the fiber of $p$ over $e_1$. It follows that the homotopy classes $[\sigma]$ and $[\rho]$ of $\sigma$ and $\rho$, respectively, generate $\pi_3(\SO_4)$ and we have
\begin{equation}\label{eq:pi_3(SO_4)}
\pi_3(SO_4)=\Z[\sigma]\oplus\Z[\rho]=\Z\oplus\Z.
\end{equation}

Below we will always use the isomorphism in \eqref{eq:pi_3(SO_4)} when we consider Smale invariants: we write simply $\Omega(f)=(a,b)\in\Z\oplus\Z$ if $\Omega(f)=a[\sigma]+b[\rho]\in\pi_3(\SO_4)$. Furthermore, we will use throughout the definition of the Smale invariant $\Omega$ as given in \cite[\S1]{hughes}.
 
We now turn to the more precise description of $\Omega(f)$ in terms of $D(f)$ and $H(f)$. According to \cite[Theorem~3.1]{hughes} (see also Remark \ref{rem:signsk-m} below), if $f\co S^3\to\R^4$ is an immersion which extends to an immersion $F\co V\to\R^4$ of a compact oriented $4$-manifold $V$ with boundary $\pa V=S^{3}$ then
\begin{equation}\label{eq:hughescob}
\Omega(f)=\left(\chi(V)-1,\frac{3\sigma(V)-2\chi(V)+2}{4}\right)\in\Z\oplus\Z,
\end{equation}
where $\chi(V)$ denotes the Euler characteristic of $V$ and $\sigma(V)$ its signature. Applying \cite[Theorems~2.2(b) and 2.5(b)]{k-m} to the stable framing induced by $f$, we find $\ndeg(f)=\chi(V)$ and $\hdef(f)=-3\sigma(V)$ and we conclude that for $f$ as above
\begin{equation}\label{eq:smale}
\Omega(f)=
\left(\ndeg(f)-1,\frac{-\hdef(f)-2(\ndeg(f)-1)}{4}\right)\in\Z\oplus\Z.
\end{equation}
On the other hand \cite[Theorems~2.2(a) and 2.5(a)]{k-m} imply that the formula \eqref{eq:smale} gives a homomorphism $\Imm[S^{3},\R^{4}]\to\Z\oplus\Z$, see Remark \ref{rem:signsk-m}. 

We claim that \eqref{eq:smale} holds for any immersion. To see this we will use two subgroups $E$ and $N$ of $\Imm[S^{3},\R^{4}]$. Let $\iota\colon \R^{4}\to\R^{5}$ denote the natural inclusion. The subgroup $E$ (respectively $N$) consists of the regular homotopy classes of immersions $f\colon S^{3}\to\R^{4}$ such that $\iota\circ f\colon S^{3}\to\R^{5}$ is regularly homotopic to an embedding (respectively to the standard embedding $\iota\circ s$, where $s\colon S^{3}\to\R^{4}$ is the standard embedding of the sphere). By \cite[Lemmas 3.3.3 and 7.1.1, Proposition 4.1.2]{ekholm}, see also \cite{hughes-melvin}, the subgroup $E\subset\Imm[S^{3},\R^{4}]$ has index 24 and the subgroup $N$ is generated by the regular homotopy class of the orientation reversed standard sphere $\hat s\colon S^{3}\to\R^{4}$.

We next show that \eqref{eq:smale} holds on $E$. Let $f\colon S^{3}\to\R^{4}$ be an immersion such that $\iota\circ f$ is regularly homotopic to an embedding  $\tilde g\colon S^{3}\to\R^{5}$. Then $\tilde g$ admits an embedded orientable Seifert surface $\tilde G\colon W\to\R^{5}$. Since $\tilde G$ admits a normal vector field, the Hirsch lemma, see \cite{hirsch}, implies that it can be pushed down to an immersion $G\colon W\to\R^{4}$ by regular homotopy. (The manifold to which the Hirsch lemma is applied has boundary and is hence a regular neighborhood of its $3$-skeleton. Therefore the Hirsch lemma applies even though the dimension of the source equals the dimension of the target.) Restricting $G$ to the boundary $\pa W=S^{3}$ we get an immersion $g\colon S^{3}\to\R^{4}$ such that, $\iota\circ f$ and $\iota\circ g$ are regularly homotopic. Thus, the regular homotopy classes of $f$ and $g$ differ by an element of $N$. As mentioned above, $N$ is generated by the class of $\hat s$. Therefore there exists an integer $k$ such that
\begin{equation}\label{eq:lifteq}
\Omega(f)=\Omega(g)+k\,\Omega(\hat s).
\end{equation}
(See \cite[Proposition 7.1.2]{ekholm} for the geometric counterpart of \eqref{eq:lifteq}.) Now, both $g$ and $\hat s$ extend to immersions of compact oriented $4$-manifolds and thus \eqref{eq:smale} holds for $g$ and $\hat s$. As mentioned above, \eqref{eq:smale} gives a homomorphism and consequently \eqref{eq:lifteq} implies that \eqref{eq:smale} holds for $f$. We conclude that \eqref{eq:smale} holds on the finite index subgroup $E$ and hence in general.    

\begin{rem}\label{rem:signsk-m}
In \cite[Theorems~2.2(a) and 2.5(a)]{k-m} it is described how the counterparts $d(\phi)$ of $D(f)$ and $h(\phi)$ of $H(f)$ change as the (stable) trivialization $\phi$ changes: acting by $\sigma$ on $\phi$ decreases $d$ by $1$ and increases $h$ by $2$, and acting by $\rho$ does not change $d$ and increases $h$ by $4$. Here we keep the trivialization fixed but vary the framing. Acting on the trivialization by $\sigma$ (or $\rho$) corresponds to acting on the framing by $\sigma^{-1}$ (or $\rho^{-1}$). Consequently, changing the Smale invariant by $\sigma$ increases $D$ by $1$ and decreases $H$ by $2$, and changing it by $\rho$ keeps $D$ fixed and decreases $H$ by $4$. This explains the sign differences between the formulas in \cite[Theorems~2.2(a) and 2.5(a)]{k-m} and the formulas for the components in the right hand side of \eqref{eq:smale}. Also, it shows that $H(f)+2D(f)-2$ is divisible by $4$ for any immersion $f$, compare \cite[Theorem 2.6]{k-m}.

The coefficient $\frac{3}{4}$ of $\sigma(V)$ in the second component in the right hand side of \eqref{eq:hughescob} differs from the corresponding coefficient in \cite[Theorem~3.1]{hughes} which is $-\frac{3}{4}$. The reason for this sign difference can be traced to the proof of the result given in \cite{hughes}: on p.\,178 of that paper it is written ``This is just $-\frac{1}{2}p_1(M)$.'', thereby relating the obstruction to extending a trivialization of the tangent bundle of an almost parallelizable $4$-manifold with $-\frac{1}{2}p_1(M)$. With our conventions this obstruction should be related to $+\frac{1}{2}p_1(M)$.     
\end{rem}

\subsection{The Smale invariant in $5$-space and cobordism class}\label{subsect:cobordism}
The long exact sequence of homotopy groups of the fibration
\[
\begin{CD}
SO_{2}\to SO_{5}\to V_{5,3},
\end{CD}
\]
where $V_{5,3}$ is the Stiefel manifold of orthogonal $3$-frames in $\R^{5}$, shows that $\pi_{3}(SO_{5})=\pi_{3}(V_{5,3})$. We can thus consider the Smale invariant of an immersion $h\colon S^{3}\to\R^{5}$ as an element
\[
\Omega(h)\in\pi_{3}(SO_{5})\approx\Z.
\]
In particular if $f\colon S^{3}\to\R^{4}$ is an immersion and if $\iota\colon\R^{4}\to\R^{5}$ denotes the inclusion then the fact that the map $\pi_3(\SO_4)\to\pi_3(\SO_5)$ induced by the inclusion is given by $(a,b)\mapsto a+2b$, see e.g.~\cite[p.\,178]{hughes}, implies that the Smale invariant of $\iota\circ f$ satisfies
\[
\Omega(\iota\circ f)=-\frac{1}{2}\hdef(f)\in\Imm[S^3,\R^5]\approx\Z.
\]

As mentioned in Section \ref{sec:intr}, the oriented cobordism group $\imm(3,4)$ of immersions of oriented $3$-manifolds into $\R^4$ is isomorphic to the stable homotopy group $\pi_3^S\approx\Z_{24}$ of spheres. The isomorphism can be constructed  via the Pontrjagin-Thom construction and in order to compute the cobordism class represented by an immersion $f\colon S^{3}\to\R^{4}$
we use the following factorization of the map $\Imm[S^3,\R^4]\to\imm(3,4)$:
\[
\begin{CD}
\pi_3(\SO_4) @>>> \pi_3(\SO_5)
@>>>
\cdots
@>>>
\pi_3(\SO)
@>{J}>>
\pi_3^S\approx\Z_{24},
\end{CD}
\]
where the last map is the $J$-homomorphism which is surjective (see \cite[p.\,180]{hughes}).
In fact, we get an isomorphism
\[
\imm(3,4)\to\Z_{24}
\]
as follows. If $f\colon M\to\R^{4}$ is an immersion of a closed $3$-manifold representing a cobordism class of immersions then we map it to 
\[
-\frac{\hdef(f)+3\mu(M,s_f)}{2}\pmod{24}, 
\]
where $\mu(M,s_f)\in\Z_{16}$ is the $\mu$-invariant of $M$ with respect to the spin structure
$s_f$ induced by $f$, see \cite{takase}. Note that $\mu(S^3,s_f)$ is always equal to zero in $\Z_{16}$.

\subsection{Singular Seifert surfaces and Stingley's index}\label{subsect:stingley}
In Section \ref{subsect:regular}, an expression for the Hirzebruch defect of an immersion $S^{3}\to\R^{4}$ which does extend to an immersion of a compact oriented $4$-manifold was obtained, see \eqref{eq:hughescob} and \eqref{eq:smale}. In this section we generalize that result to immersions $f\co S^3\to\R^4$ which may not extend to immersions
of compact oriented $4$-manifolds. For such immersions, the Hirzebruch defect can be read off from extensions by singular maps which are regular near the boundary.  An extension by a smooth map with singularities of a given map (embedding, immersion, generic map, etc.) is often called a singular Seifert surface because of similarities with Seifert surfaces in knot theory. For immersions, several variations of the original definition have been used since the notion was introduced in \cite{e-s0}.

In \cite{takase}, generic maps $F\co V\to\R^4$ of compact oriented $4$-manifolds were studied and the following result was derived: If $\partial V= S^{3}$ and the restriction $f=F|_{\partial V}\colon S^{3}\to\R^{4}$ is an immersion, and if $F$ is non-singular in a neighborhood of $\pa V$, then the Hirzebruch defect $\hdef(f)$ satisfies
\begin{equation}\label{eq:hdef1}
\hdef(f)=-3\sigma(V)-\sharp\Sigma^{2,0}(F),
\end{equation}
where $\sharp\Sigma^{2,0}(F)$ is the (algebraic) number of
umbilic points of $F$.
This was essentially proved in \cite[Proof of Lemma~3.2]{takase}. Here it should be mentioned that the sign convention in \cite{takase} differs from that used here; see Remark~\ref{rem:errata} below.
The proof of \eqref{eq:hdef1} is based on the Thom polynomial for umbilic singularities of a generic map
$G\co M\to\R^4$ of a closed oriented $4$-manifold $M$:
\begin{equation}\label{eq:thom}
\sharp\Sigma^{2,0}(G)=-\bigl\langle p_1(M),[M]\bigr\rangle=-3\sigma(M),
\end{equation}
where $p_1$ denotes the first Pontrjagin class and $[M]\in H_{4}(M;\Z)$ is the fundamental class, see \cite{macpherson,ronga}, \cite[\S2.1 in Chapter~4]{avgl}, and Remark \ref{rem:errata} below.

For the purposes of this paper, we weaken the condition for singular Seifert surfaces and use the following characterization.

\begin{defn}\label{defn:sss}
Let $f\co S^3\to\R^4$ be an immersion.
A \textit{singular Seifert surface} for $f$ is a smooth map
$F\co V\to\R^4$
from a compact oriented $4$-manifold $V$ with $\partial V=S^{3}$ which satisfies the following conditions. The restriction $F|_{\partial V}$ to the boundary equals $f$, there is a neighborhood of $\partial V$ where the map $F$ has no singularity, for any $p\in V$, the rank $\rank(dF_{p})$ of the differential
\[
dF_{p}\colon T_{p}V\to T_{F(p)}\R^{4}
\]
satisfies $\rank(dF_p)\ge 2$, and points $q$ with $\rank(dF_q)=2$ are isolated.
\end{defn}

Note that if $F\colon V\to\R^{4}$, $\partial V= S^{3}$ is a generic (i.e., locally stable) map which is non-singular near the boundary then its possible singularities are Morin singularities where the differential has rank $3$ and isolated umbilic points where the differential has rank $2$. Therefore a singular Seifert surface in the sense of \cite[\S2.1]{takase} fulfills the requirements of Definition \ref{defn:sss}.

Smooth maps between oriented $4$-manifolds with differential of rank $2$ at isolated points and rank $>2$ elsewhere
were studied by Stingley \cite{stingley}.
He associated an index \cite[\S2]{stingley} to isolated points where the rank equals $2$ as follows. (See
Remark~\ref{rem:errata} below for details on sign conventions.) Let $g\colon M\to N$ be a smooth map between oriented $4$-manifolds. Let $p\in M$ be such that $\rank(dg_p)=2$ and assume that $p$ is an isolated rank $2$ point in the sense that there exists a neighborhood $U\subset M$ of $p$ such that $\rank(dg_{q})>2$ for all $q\in U-\{p\}$. Then there are local coordinates $x=(x_1,x_2,x_3,x_4)\in\R^{4}$ around $p$ and $y=(y_1,y_2,y_3,y_4)\in\R^{4}$ around $g(p)$ such that, in these coordinates the map $g(x)=\bigl(g_1(x),g_2(x),g_3(x),g_4(x)\bigr)$ is given by
\[
\begin{cases}
g_1(x)&=x_1,\\
g_2(x)&=x_2,\\
g_3(x)&=A(x),\\
g_4(x)&=B(x),
\end{cases}
\]
where $A(0)=B(0)=0$ and where
\[
\left(
\begin{matrix}
\frac{\pa A}{\pa x_3} & \frac{\pa A}{\pa x_4}
\vspace{.2cm}\\
\frac{\pa B}{\pa x_3} & \frac{\pa B}{\pa x_4}
\end{matrix}
\right)
\]
vanishes for $x=0$ but not for $x\ne 0$.

\begin{defn}\label{defn:index}
The \textit{index} $\ind_p(g)$ of an isolated rank $2$ point $p$ as above is
\[
\ind_p(g)=\deg_{\epsilon}(\hat{g}).
\]
Here $\hat g(x)=\bigl(\hat g_1(x),\hat g_2(x),\hat g_3(x),\hat g_4(x)\bigr)$ is defined as
\[
\begin{cases}
\hat g_1(x)&=\frac{\pa A}{\pa x_3},\vspace{.1cm}\\
\hat g_2(x)&=\frac{\pa B}{\pa x_3},\vspace{.1cm}\\
\hat g_3(x)&=\frac{\pa A}{\pa x_4},\vspace{.1cm}\\
\hat g_4(x)&=\frac{\pa B}{\pa x_4},
\end{cases}
\]
in a neighborhood of $0$ and $\deg_{\epsilon}(\hat{g})$ is the degree of $\hat g$ around $0$. I.e.~$\deg_{\epsilon}(\hat{g})$ denotes the signed sum of preimages at any sufficiently small regular value $\epsilon$ of $\hat g$. 
\end{defn}

It is straightforward to check that $\ind_{p}(g)$ is independent of the choice of local coordinates. Furthermore, the index of an umbilic point of a map between $4$-manifolds equals $\pm1$ and in this case we use the index to define the sign of umbilic points.

\begin{rem}\label{rem:ori}
If the orientation of the source manifold in Definition~\ref{defn:index} is changed then
the index of each rank $2$ point changes its sign. On the other hand, if the orientation of the target manifold is changed then the index of any rank $2$ point remains unchanged.
\end{rem}

For a smooth map $g\colon M\to N$ between oriented $4$-manifolds where $M$ is closed, let
\[
\Sigma^{2}(g)=\{p\in M\ST \rank(dg_p)=2\}
\]
and note that if the rank $2$ points of $g$ are isolated then $\Sigma^2(g)$ is finite. For such maps $g$ with isolated rank $2$ points we define \textit{the algebraic number of rank $2$ points} $\sharp\Sigma^{2}(g)$ of $g$ as
\[
\sharp\Sigma^2(g)=\sum_{p\in\Sigma^2(g)}\ind_p(g).
\]

Using the above notation, there is the following Riemann-Hurwitz type formula for closed oriented $4$-manifolds.

\begin{thm}
[(Stingley {\cite[Theorem~I]{stingley}}, see also Porter \cite{porter} and Harvey and Lawson {\cite[\S8]{h-l}})]\label{t:RieHur}
Let $g\colon M\to N$ be a smooth map between oriented $4$-manifolds
with isolated rank $2$ points and assume that $M$ is closed. Then
\begin{equation}\label{eq:sigma2long}
\bigl\langle(g^*p_1(N)-p_1(M)),[M]\bigr\rangle=\sharp\Sigma^2(g),
\end{equation}
where $p_1$ denotes the first Pontrjagin class and $[M]$ is the fundamental class of $M$.
\end{thm}

In particular, Theorem \ref{t:RieHur} implies that if
$g\co M\to\R^4$ is a map from a closed oriented $4$-manifold $M$ with isolated rank $2$ points, then
\begin{equation}\label{eq:stingley}
\sharp\Sigma^2(g)=-\bigl\langle p_1(M),[M]\bigr\rangle=-3\sigma(M),
\end{equation}
where $\sigma(M)$ is the signature of $M$.

\begin{rem}\label{rem:errata}
The Thom polynomial \eqref{eq:thom} and Stingley's formula \eqref{eq:stingley} imply that for an appropriate choice of signs, the count of indices of rank $2$ points of maps of a closed $4$-manifold $M$ into $\R^{4}$ gives the first Pontrjagin number of $M$. However, the sign conventions for the index which appear in the literature are often conflicting (see \cite{macpherson,porter,ronga,s-s,stingley,takase} and also \cite[\S2.2 in Chapter~4]{avgl}, for example). Even within Stingley's paper \cite{stingley}, signs appear to be in conflict: compare the signs in \cite[Lemma~2.1]{stingley} to those in the first line of the proof of \cite[Lemma~4.1]{stingley}.

In order to check which sign in \eqref{eq:sigma2long} is compatible with our conventions for $\sharp\Sigma^{2}$ in Definition \ref{defn:index} it is sufficient to consider one example where the count is non-zero. To this end we use the following, see \cite[p.\,399]{porter}. Let $(z_1,z_2,z_3)$ be coordinates on $\C^{3}$, let $S^{5}\subset \C^{3}$ denote the unit sphere, and consider $\C P^{2}=S^{5}/U_1$, where $e^{i\theta}\in U_1$ acts by scalar multiplication. We give $\C P^{2}$ the complex orientation so that its signature satisfies $\sigma(\C P^{2})=+1$.

Write $[z_1,z_2,z_3]$ for the point in $\C P^{2}$ represented by $(z_1,z_2,z_3)\in S^{5}$. Let $f\colon \C P^{2}\to\C\times\R^{2}=\R^{4}$ be given by 
\[
f([z_1,z_2,z_3])=\bigl(z_1\bar z_2,\Re(z_1\bar z_3),\Re(z_2\bar z_3)\bigr).
\]   
Then $f$ has isolated $\Sigma^{2}$-singularities at the following seven points in $\C P^{2}$:
\[
p_0=[0,0,1],\,\,p_{\pm 1}=\left[0,\frac{1}{\sqrt{2}},\pm\frac{1}{\sqrt{2}}\right]\!\!,\,\,
p_{\pm 2}=\left[\frac{1}{\sqrt{2}},0,\pm\frac{1}{\sqrt{2}}\right]\!\!,\,\, 
p_{\pm 3}=\left[\frac{1}{\sqrt{2}},\pm\frac{1}{\sqrt{2}},0\right]\!\!.
\]

Recall that $\ind_p(f)$ denotes the contribution from the singular point $p\in\Sigma^{2}(f)$ to $\sharp\Sigma^2(f)$. The orientation preserving diffeomorphisms of $\C P^{2}$: $l[z_1,z_2,z_3]=[z_2,z_1,z_3]$ and $k_{\pm}[z_1,z_2,z_3]=[\mp z_1,\pm z_2,z_3]$, and the diffeomorphisms of $\R^{4}$: $i(\eta_1,\eta_2,\eta_3,\eta_4)=(\eta_1,-\eta_2,\eta_4,\eta_3)$ and $j_{\pm}(\eta_1,\eta_2,\eta_3,\eta_4)=(-\eta_1,-\eta_2,\mp \eta_3,\pm \eta_4)$ satisfy $i\circ f\circ l=f$ and $j_\pm\circ f\circ k_{\pm}=f$. Hence
\begin{equation}\label{eq:relationsbtwnind}
\ind_{p_{+1}}(f)=\ind_{p_{-1}}(f)=\ind_{p_{-2}}(f)=\ind_{p_{+2}}(f)\quad\text{ and }\quad
\ind_{p_{-3}}(f)=\ind_{p_{+3}}(f),
\end{equation}
see Remark \ref{rem:ori}.  

We next compute $\ind_{p_0}(f)$. Note that $(u_1,v_1,u_2,v_2)$ where $z_j=u_j+iv_j$, $j=1,2$ correspond to complex coordinates near $[0,0,1]\in\C P^{2}$. Taking coordinates $(x_1,x_2,x_3,x_4)=(u_1,u_2,v_2,v_1)$ on $\C P^{2}$ we find, with notation as in Definition \ref{defn:index}, that 
\[
A(x)=u_1u_2+v_1v_2=x_1x_2+x_3x_4 \text{ and } B(x)=-u_1v_2+u_2v_1=-x_1x_3+x_2x_4.
\]
Thus,  
\[
\hat g(x)=\left(\frac{\pa A}{\pa x_3},\frac{\pa B}{\pa x_3},\frac{\pa A}{\pa x_4},\frac{\pa B}{\pa x_4}\right)=
(x_4,-x_1,x_3,x_2)
\] 
and $\ind_{p_0}(f)$, which is defined as the local degree of $\hat g$ at $0$, equals $-1$. Using similar quadratic local coordinate expressions near $p_k$, $k\in\{\pm1,\pm2,\pm3\}$, one finds that 
\begin{equation}\label{eq:locdeg=1}
\ind_{p_k}(f)=\pm 1,\text{ for } k\in\{\pm1,\pm2,\pm3\}. 
\end{equation}
By \eqref{eq:sigma2long}, taken up to sign, $\sharp\Sigma^{2}(f)=\pm\sigma(\C P^{2})=\pm 3$. Using \eqref{eq:relationsbtwnind} and \eqref{eq:locdeg=1}, we find that $\ind_{p_{-1}}(f)=-\ind_{p_{-3}}(f)$ and that $\sharp\Sigma^{2}(f)=3\ind_{p_0}(f)$.
Consequently, $\sharp\Sigma^{2}(f)=-3=-3\sigma(\C P^{2})$ and we conclude that the signs in \eqref{eq:sigma2long} and \eqref{eq:stingley} are correct.  
\end{rem}

\subsection{A formula for the Hirzebruch defect}\label{subsect:hirzebruch}
Let $f\colon S^{3}\to\R^{4}$ be an immersion and let $F\colon V\to \R^{4}$ be a singular Seifert surface of $f$, see Definition~\ref{defn:sss}. The following result is a slight generalization of \eqref{eq:hdef1}.
\begin{prop}\label{prop:hirzebruch}
The Hirzebruch defect $\hdef(f)$ (of the stable framing induced by $f$) satisfies
\begin{equation}\label{eq:hirzebruch}
\hdef(f)=-3\sigma(V)-\sharp\Sigma^2(F).
\end{equation}
In particular, the right hand side of \eqref{eq:hirzebruch} does not depend on the choice of singular Seifert surface $F\colon V\to\R^{4}$ extending $f$.
\end{prop}

\begin{proof}
Let $F'\co V'\to\R^4$ be a singular
Seifert surface for an immersion $f'\co S^3\to\R^4$ which is regularly homotopic to $f$.
Using $F'$, $F$, and
the trace of a regular homotopy $h\co S^3\times[0,1]\to\R^4\times[0,1]$ between $f$ and $f'$,
we construct a map
\[
\hat G\colon V\cup_\partial S^3\times[0,1]\cup_\partial(-V')\to\R^4\times[0,1]
\]
from the closed oriented $4$-manifold $W=V\cup_\partial S^3\times[0,1]\cup_\partial(-V')$,
where $-V'$ denotes the manifold $V'$ with reversed orientation, which is smooth after smoothing corners. Since the maps $F$ and $F'$ are immersions near the boundary and since the trace of a regular homotopy is an immersion, it follows that $\hat G$ is an immersion in a neighborhood of $S^3\times[0,1]$. Since the composition of the projection  $\R^{4}\times[0,1]\to \R^{4}$ with an immersion does not have any rank $2$ points, we find that the map $G\colon W\to\R^{4}$ which equals $\hat G$ composed with the projection to $\R^4$ has isolated rank 2 points which are in natural 1-1 correspondence with those of $F$ and $F'$. In particular, the algebraic number of rank $2$ points $\sharp\Sigma^2(G)$ satisfies
\[
\sharp\Sigma^2(G)=\sharp\Sigma^2(F)-\sharp\Sigma^2(F').
\]

By \eqref{eq:stingley}, we have $\sharp\Sigma^2(G)=-3\sigma(W)$ and hence
\[
-3\sigma(V)-\sharp\Sigma^2(F)=-3\sigma(V')-\sharp\Sigma^2(F')
\]
by Novikov additivity. It follows that
\[
-3\sigma(V)-\sharp\Sigma^2(F)
\]
depends only on the regular homotopy class of $f$. We can thus define a homomorphism $a\co\Imm[S^3,\R^4]\to\Z$ as follows. For $[f]\in\Imm[S^3,\R^4]$, pick a representative immersion $f$ of $[f]$ and a singular Seifert surface $F\colon V\to \R^{4}$ of $f$ and define
\[
a([f])=-3\sigma(V)-\sharp\Sigma^2(F).
\]

Recall the index 24 subgroup $E\subset\Imm[S^3,\R^4]$ of regular homotopy classes of immersions which when composed with the inclusion into $\R^5$ admit a regular homotopy to an embedding, see Section \ref{subsect:regular}. By \cite[Lemma 3.3.3, Proposition 4.1.2]{ekholm} we can write this subgroup as
\[
E=\bigl\{[f]\in\Imm[S^{3},\R^{4}]\ST \Omega(f)=(m,n)\in\Z\oplus\Z,\,\, m+2n\in 24\Z\bigr\},
\]
where $f$ denotes a representative for the class $[f]$ and where $\Omega$ is the Smale invariant. 

To finish the proof we show that the homomorphism $a$ agrees with the Hirzebruch defect. Consider a smooth degree $4$ hypersurface in the complex projective space $\C P^3$. Such a surface is a K3-surface which is a closed simply connected $4$-manifold $X$ with second Betti number $22$ and signature $\pm 16$ depending on orientation. If $X_0$ denotes the complement of an open disk in $X$ then $X_0$ is parallelizable and hence it immerses into $\R^{4}$. Let $F\colon X_0\to\R^4$ be an immersion and let $f=F|_{\pa X_0}$ be the induced immersion of $S^3$. As above, by \cite{hughes} the Smale invariant of $f$ (with $X_0$ oriented so that its signature equals $16$) equals     
\[
\left(\chi(X_0)-1,\frac{3\sigma(X_0)-2\chi(X_0)+2}{4}\right)=
\left(22,\frac{3\cdot 16-2\cdot 23+2}{4}\right)
=(22,1).
\]
This in combination with \cite[Lemma 7.1.1]{ekholm}, imply that $[f]$ and the regular homotopy class $[\hat s]$ of the orientation reversed standard embedding $\hat s\colon S^{3}\to\R^{4}$, together generate $E$, see Section \ref{subsect:regular}. (Recall that addition corresponds to connected sum of immersions, see \cite[Section 2]{kervaire} for the definition and for additivity of the Smale invariant.) It follows from \cite[Theorem~2.5(b)]{k-m} that the homomorphism $a$ coincides with the Hirzebruch defect on $E$ and hence the two homomorphisms agree on all of $\Imm[S^3,\R^4]$.
\end{proof}

\section{Explicit constructions of singular Seifert surfaces}\label{sect:main}
In this section we first construct singular Seifert surfaces for the immersions $S^{3}\to\R^{4}$ described in Section \ref{sec:intr} and then use these to prove Theorem \ref{t:main}.

We will use the following notation throughout this section. The $2$-plane bundle over $S^{2}$ of Euler number $k$ will be denoted by $\xi_k$, $E(\xi_k)$ will denote the total space of the corresponding $2$-disk bundle, and $L(k,1)$ the total space of the corresponding circle bundle ($L(k,1)$ is a $(k,1)$-lens space). Furthermore, recall from Section \ref{sec:intr} that the normal bundle of a self-transverse immersion $f\colon S^{2}\to\R^{4}$ with algebraic number of double points equal to $-n$ is $\xi_{2n}$. Thus, micro-extension gives an immersion $F_{2n}\colon E(\xi_{2n})\to\R^{4}$ which agrees with $f$ when restricted to the $0$-section.
Precomposing with the universal covering $\pi_{2n}\colon S^{3}\to L(2n,1)=\pa E(\xi_{2n})$, we get the immersions
\[
g_n=F_{2n}\circ\pi_{2n}\colon S^{3}\to\R^{4}, \quad n>0.
\]

We next construct natural fillings of the maps $g_n$. To this end, note that the space $L(1,1)=\partial{E(\xi_1)}$ is $S^3$, thought of as the total space of the Hopf fibration, and that $E(\xi_1)$ is $\C P^{2}_{0}$, the complement of an open disk in the complex projective plane. Consider the $k$-fold branched cover
\[
\Pi_{k}\colon E(\xi_1)\to E(\xi_k)
\]
which extends the universal cover
\[
\pi_k\colon S^{3}=\pa E(\xi_1)\to\pa E(\xi_k)=L(k,1),
\]
and which is the $k$-fold branched cover with a single branch point at the origin in each fiber disk. The map $G_{n}\colon \C P^{2}_{0}\to\R^{4}$,
\[
G_n = F_{2n}\circ \Pi_{2n}
\]
then extends the immersion $g_n\colon S^{3}\to\R^{4}$ but is {\em not} a singular Seifert surface for $g_n$ for the following reason. The rank of $dG_n$ equals $4$ outside the $0$-section in $E(\xi_1)$ (in $\C P^{2}_{0}-\C P^{1}$) and equals $2$ at any point on the $0$-section (along $\C P^{1}\subset \C P^{2}_{0}$) and thus its rank $2$ points are not isolated.

The rank $2$ points of $G_n$ forming a submanifold is reminiscent of the critical locus of a Morse-Bott function. In analogy with the Morse-Bott case, we will construct a perturbation of $G_n$ below which creates a map with exactly two isolated rank 2 points. Furthermore, we will compute the indices of these rank 2 points. In fact, our perturbation of $G_n$ has the form $F_{2n}\circ\Pi_{2n}^{\epsilon}$ where $\Pi_{2n}^{\epsilon}$ is a perturbation of the branched cover $\Pi_{2n}$. In order to construct the perturbation, we use the following local coordinate description of the bundles $\xi_k$.

Consider the decomposition of $S^{2}$ into two disks, the northern- and southern hemispheres $D_{N}$ and $D_{S}$, respectively. After identification of these disks with the unit disk  $D=\{w\in\C\ST|w|\le1\}$ in the complex plane $\C$, we have
\[
S^2=D_S\cup_\psi D_N,
\]
where $\psi\co\partial{D_S}\to\partial{D_N}$ is the map given by $\psi(e^{i\theta})=e^{-i\theta}$. The total spaces $E(\xi_k)$ can then be described as follows:
\[
E(\xi_k)= D_S\times D\cup_{\psi_k} D_N\times D,
\]
where $\psi_k\colon \pa D_S\times D\to\pa D_N\times D$ is given by
\[
\psi_k(e^{i\theta},z)=(e^{-i\theta},e^{-ki\theta}z).
\]
To see this, it is sufficient to note that the clutching function described gives a bundle with first Chern class (Euler number) equal to $k$.

In these local coordinates the map $\Pi_{k}\colon E(\xi_1)\to E(\xi_k)$ is the following:
\[
\Pi_k(w,z)=
\begin{cases}
(w,z^{k})\in D_S\times D &\text{if }(w,z)\in D_S\times D,\\
(w,z^{k})\in D_N\times D &\text{if }(w,z)\in D_N\times D.
\end{cases}
\]
To see that this formula defines a map as claimed, note that on the overlap $(e^{i\theta},z)\in \pa D_S\times D\subset E(\xi_1)$ is identified with $(e^{-i\theta},e^{-i\theta}z)\in \pa D_N\times D\subset E(\xi_1)$. The former maps to $(e^{i\theta},z^{k})\in \pa D_S\times D\subset E(\xi_k)$ which is identified with $(e^{-i\theta},e^{-ki\theta}z^{k})\in \pa D_N\times D\subset E(\xi_k)$ which in turn is the image of the latter.

In order to perturb $\Pi_{k}\colon E(\xi_1)\to E(\xi_k)$ we consider the restriction of this map to fibers: $h(z)=z^{k}$. Let $c$ be any complex number and let $h_{c}(z)=z^{k}+\bar z c$.
\begin{lem}\label{lem:2dimmap}
The singular set of $h_{c}$ is the circle
\[
\Sigma=\bigl\{z\ST |z|=(k^{-1}|c|)^{\frac{1}{k-1}}\bigr\}.
\]
If $c\ne 0$, then the rank of $dh_{c}$ equals $1$ along $\Sigma$ and the kernel field $\ker(dh_{c})$ is transverse to $\Sigma$ except at $k+1$ cusp points where the kernel field has simple tangencies with $\Sigma$. In particular, $h_{c}$ is a locally stable map for $c\ne 0$.
\end{lem}

\begin{proof}
If $\xi$ is a tangent vector to $\C$ thought of as a complex number we have
\[
dh_{c}(\xi)=k z^{k-1}\xi+c\bar \xi.
\]
Consequently, $dh_{c}$ has kernel at points $z$ where the equation $k z^{k-1}\xi+c\bar \xi=0$ has non-trivial solutions. This is the case if
\begin{equation}\label{e:kernel}
k z^{k-1}=-c\bar\xi/\xi,
\end{equation}
or in other words when $z\in\Sigma$. The properties of the kernel field are straightforward consequences of \eqref{e:kernel}.
\end{proof}

We next deform the map $\Pi_k$ using the perturbation described above on fibers. Note however that $c$ in the above discussion must then be replaced by the image of a section of $E(\xi_1)$ and hence we cannot avoid rank $2$ points.

Fix $\epsilon\in(0,\tfrac14)$. Define $\Pi^{\epsilon}_{k}\colon E(\xi_1)\to E(\xi_k)$ as follows, using the local coordinates introduced above,
\begin{equation}\label{e:Pie}
\Pi^{\epsilon}_k(w,z)=
\begin{cases}
\left(w,(1-\epsilon)z^{k}+\epsilon w\bar z \right)\in D_{S}\times D
&\text{if }(w,z)\in D_S\times D,\\
\left(w,(1-\epsilon)z^{k}+\epsilon w^{k}\bar z \right)\in D_{N}\times D
&\text{if }(w,z)\in D_N\times D.
\end{cases}
\end{equation}
To see that \eqref{e:Pie} indeed defines a map, note that for $(w,z)=(e^{i\theta},z)\in\pa D_S\times D$,
\[
\psi_{k}\left(e^{i\theta},(1-\epsilon)z^{k}+\epsilon e^{i\theta}\bar z \right)=
\left(e^{-i\theta},(1-\epsilon)e^{-ik\theta}z^{k}+\epsilon e^{-i(k-1)\theta}\bar z\right)\in \pa D_{N}\times D,
\]
which is the image of the corresponding point $\psi_1(w,z)=(e^{-i\theta},e^{-i\theta} z)\in\pa D_N\times D$ according to the second row in the right hand side of \eqref{e:Pie}. Furthermore if $|z|\le 1$ and $|w|\le 1$ then $|(1-\epsilon)z^{k}+\epsilon w\bar z|\le 1$ so that second components in the vectors in the right hand side of \eqref{e:Pie} lie in $D$.

\begin{lem}\label{lem:sourceimm}
The map $\Pi^{\epsilon}_{k}\colon \C P^{2}_{0}\to E(\xi_k)$ is a map with two isolated rank 2 points: $p_S$ corresponding to $(0,0)\in D_{S}\times D$ and $p_N$ corresponding to $(0,0)\in D_{N}\times D$. Furthermore, $\ind_{p_S}(\Pi_{k}^{\epsilon})=k-1$ and
$\ind_{p_N}(\Pi_{k}^{\epsilon})=k(k-1)$.
\end{lem}

\begin{proof}
If a tangent vector to $D_S\times D$ or $D_N\times D$ is viewed as a pair of complex numbers $(\eta,\xi)$ and if $\epsilon'=\frac{\epsilon}{1-\epsilon}$ then, with $h_c$ as in Lemma \ref{lem:2dimmap},
\begin{align*}
&\bigl[d\Pi^{\epsilon}_k(w,z)\bigr](\eta,\xi)=\\
&\begin{cases}
\left(\eta,(1-\epsilon)\bigl[dh_{\epsilon' w}(z)\bigr](\xi)+\epsilon\bar z\eta \right)\in T(D_{S}\times D)
&\text{if }(w,z)\in D_S\times D,\\
\left(\eta,(1-\epsilon)\bigl[dh_{\epsilon' w^{k}}(z)\bigr](\xi)+\epsilon kw^{k-1}\bar z \eta \right)\in T(D_{N}\times D)
&\text{if }(w,z)\in D_N\times D.
\end{cases}
\end{align*}
Lemma \ref{lem:2dimmap} then implies that $p_S$ and $p_N$ are the only rank $2$ points of $\Pi^{\epsilon}_{k}$.

To compute the indices we consider the map $T\colon \C^{2}\to\C^{2}$
\[
T(z,w)=\left(w,z^\ell+\bar{z}{w}^m\right),
\]
where $\ell\ge 2$ and $m\ge 1$. Clearly the only rank $2$ point of $T$ is the origin.
To compute the index $\ind_{0}(T)$, write $K(w,z)=z^\ell+\bar{z}{w}^m$, then
according to Definition~\ref{defn:index}, the index equals the degree near $0$ of the map $\kappa\colon\C^{2}\to\C^{2}$ where
\begin{align*}
\kappa(w,z)=\left(\frac{\pa K}{\pa z}+\frac{\pa K}{\pa\bar{z}},\,\,i\left(\frac{\pa K}{\pa z}-\frac{\pa K}{\pa\bar{z}}\right)\right)
=\left(\ell z^{\ell-1}+w^m,i(\ell z^{\ell-1}-w^m)\right).
\end{align*}

Note that there exists a complex linear isomorphism $L$ of $\C^{2}$ such that
\begin{align*}
L\circ \kappa(w,z)=\left(w^m,z^{\ell-1}\right).
\end{align*}
Since $\kappa$ is holomorphic, the local degree at any inverse image of a regular value equals $+1$. Hence if $\epsilon$ is any sufficiently small regular value of $\kappa$ we have
\[
\ind_{0}(T)=\deg_{\epsilon}(\kappa)=m(\ell-1).
\]
Taking $(m,\ell)=(1,k)$ and $(m,\ell)=(k,k)$, the statements on the indices of $p_S$ respectively $p_N$ follow.
\end{proof}

\begin{proof}[of Theorem~\ref{t:main}]
Fix a small $\epsilon>0$ and
write $\pi^{\epsilon}_{k}=\Pi^{\epsilon}_{k}|_{S^{3}}$, $S^{3}=\pa E(\xi_1)$ and note that \[
\pi^{s\epsilon}_{k}\colon S^{3}\to E(\xi_k),\quad 0\le s\le 1
\]
gives a regular homotopy connecting the immersions $\pi_k\colon S^{3}\to L(k,1)\subset E(\xi_{k})$ to $\pi_{k}^{\epsilon}\colon S^{3}\to E(\xi_k)$. Consequently the map
\[
g_n^{\epsilon}=F_{2n}\circ\pi^{\epsilon}_{2n}\colon S^{3}\to\R^{4}
\]
is regularly homotopic to $g_n\colon S^{3}\to\R^{4}$. Furthermore, since $F_{2n}\colon E(\xi_{2n})\to\R^{4}$ is an immersion, Lemma \ref{lem:sourceimm} implies that
\[
G_{n}^{\epsilon}=F_{2n}\circ\Pi^{\epsilon}_{2n}\colon \C P^{2}_{0}\to\R^{4}
\]
is a singular Seifert surface of $g_{n}^{\epsilon}$ with
\[
\sharp\Sigma^{2}(G_{n}^{\epsilon})=2n-1 + 2n(2n-1)=4n^{2}-1.
\]
Proposition \ref{prop:hirzebruch} then gives the Hirzebruch defect
\[
\hdef(g_n)=\hdef(g_{n}^{\epsilon})=-3\sigma(\C P^{2}_{0})-4n^{2}+1=-4n^{2}-2.
\]
By \cite[Theorem~2.2(b)]{k-m}, the normal degree of the immersion $f_{2n}\colon \pa E(\xi_{2n})\to\R^{4}$ equals $\chi(E(\xi_{2n}))=\chi(S^2)=2$.
Since $\pi_{2n}\co S^3\to L(2n,1)$ has degree $2n$,
the composition $g_n=f_{2n}\circ\pi_{2n}$
has normal degree
\[
\ndeg(g_n)=4n
\]
(see also \cite[\S{IV}]{milnor}). Therefore, by the formula \eqref{eq:smale}, we have
\[
\Omega(g_n)=\left(4n-1,\frac{4n^2+2-2\cdot(4n-1)}{4}\right)=\left(4n-1,(n-1)^{2}\right).
\]
With this established, remaining statements are immediate consequences of the formulas in Section \ref{subsect:cobordism}.
\end{proof}

\section{Relations to other results}\label{sect:remarks}
The immersion $g_1$ was apparently first studied
by Milnor \cite[(11), \S{IV}]{milnor} and has since appeared
in several papers (e.\,g.\ \cite{l-s,ekholm}).
It was shown in \cite[Proposition~8.4.1]{ekholm} that
$g_1$ has odd Brown invariant (in $\Z_8$) and hence
represents a generator in $\Z_8$
(see also \cite[Lemma~1.7(a)]{szucs} and \cite[Remark~3.6]{takase}).
Theorem \ref{t:main} implies this result and furthermore that
the triple point invariant $\lambda(g_1)\in\Z_3$ of $g_1$ vanishes.
For a self-transverse immersion regularly homotopic to $\iota\circ g_1\co S^3\to\R^4\to\R^5$, where $\iota$ denotes the natural inclusion of $\R^{4}$ into $\R^{5}$, the triple point invariant is defined to be the linking number modulo $3$ of its image and a push-off of its double point set. (For details on this definition, see \cite[\S6.2]{ekholm}, and also \cite[\S9.1, p.\,189]{ekholm}, where the letter ``$\tau$'' should be replaced by ``$\lambda$''.) To see this, note that $\pi^{S}_{3}\approx\Z_{24}=\Z_{8}\oplus\Z_{3}$, where the first summand corresponds to the Brown invariant and the second to the linking invariant $\lambda$. 

\begin{rem}\label{rmk:linking}
The fact that $\lambda(g_1)=0$ can also be proved directly as follows. Let $f\colon S^{2}\to\R^{4}$ be an immersion with one double point of negative sign. Consider a regular homotopy of $\iota\circ f\colon S^{2}\to\R^{5}$ with support in a small neighborhood of one of the preimages $q$ of the double point of $f$ and which lifts the local sheet of $f$ around $q$ up slightly in the fifth direction. The map resulting from this deformation of $\iota\circ f$ is an embedding. We denote it $\tilde f\colon S^{2}\to\R^{5}$. Consider the map $f_2\colon L(2,1)=\R P^{3}\to\R^{4}$ obtained by micro-extension of $f$ and the composition $\iota\circ f_2\colon\R P^3\to\R^{5}$. The regular homotopy of $\iota\circ f$ induces a regular homotopy of $\iota\circ f_2$ to an embedding which we denote $\tilde f_2\colon \R P^{3}\to\R^{5}$. The image of $\tilde f_2$ lies in a small tubular neighborhood of $\tilde f(S^{2})$. 

We next note that the normal bundle of $\tilde f_2$ has a natural trivialization: one section $\nu_0$ is induced by the normal vector field $\nu$ of $f_2$ in $\R^{4}$ and the other one $\nu_1$ is induced by the fifth coordinate vector $e_5$ viewed as a normal vector of $\iota\circ f_2$. Define $\tilde g_1=\tilde f_2\circ\pi_2\colon S^{3}\to\R^{5}$, where $\pi_2\colon S^{3}\to\R P^{3}$ is the universal covering map. Then $\tilde g_1$ is regularly homotopic to $\iota\circ g_1=\iota\circ f_2\circ\pi_2$. Furthermore, $\tilde g_1$  self intersects non-transversely; in fact it is a $2-1$ map onto its image $\tilde f_2(\R P^{3})$. 

In order to perturb out of this non-transverse situation we consider $S^{3}\subset \C^{2}$ as the total space of the Hopf fibration which maps a point in $S^{3}$ to the complex line in $\C P^{1}\approx S^{2}$ it generates. In particular, Hopf fibers are intersections of $S^{3}$ with complex lines in $\C^{2}$. Let $(z_1,z_2)=(x_1+iy_1,x_2+iy_2)$ be coordinates on $\C^{2}$. By definition, the map $\tilde g_1$ satisfies $\tilde g_1(z_1,z_2)=\tilde g_1(-z_1,-z_2)$. Define the great $2$-spheres 
\[
S_0^{2}=\{(x_1+iy_1,x_2+iy_2)\in S^{3}\colon y_2=0\}\quad\text{and}\quad
S_1^{2}=\{(x_1+iy_1,x_2+iy_2)\in S^{3}\colon x_2=0\}.
\]
Let $S_{01}=S^{2}_0\cap S^{2}_1$. Then $S_{01}$ is the Hopf fiber through the point $(1,0)$. Furthermore, any other Hopf fiber $S\ne S_{01}$ satisfies 
\begin{equation}\label{eq:genericHopf}
S_{0}^{2}\cap S=\{p_0,-p_0\}\text{ and }S_1^{2}\cap S=\{p_1,-p_1\},\,\,\text{where } p_0\ne \pm p_1.
\end{equation}
We pick coordinates so that the image of the Hopf fiber $S_{01}$ under the Hopf projection lies far from the self intersection points of $f$.

Consider two functions $b_j\colon S^{3}\to\R$ with the following properties. The point $0\in\R$ is a regular value of $b_j$ and $b_j^{-1}(0)=S^{2}_j$, $j=0,1$. Then, for sufficiently small $\epsilon>0$, the map $a_{\epsilon}\colon S^{3}\to\R^{5}$,
\[
a_\epsilon(x)=\tilde g_1(x)+\epsilon \bigl(b_0(x)\nu_0(x)+b_1(x)\nu_1(x)\bigr)
\]
is a self transverse immersion regularly homotopic to $\tilde g_1$ with self intersection along $\tilde g_1(S_{01})$.

To see this, note first that $\tilde g_1$ maps distinct Hopf fibers into distinct fibers of the tubular neighborhood of $\tilde f$. Consider next a Hopf fiber $S\ne S_{01}$. It follows from \eqref{eq:genericHopf} that if $p\in S$ then $b_0(p)\ne b_0(-p)$ or $b_1(p)\ne b_1(-p)$. Consequently, $a_\epsilon$ does not have any self intersections outside $S_{01}$. Since both $b_0$ and $b_1$ vanish along $S_{01}$, it follows that $a_{\epsilon}|_{S_{01}}=\tilde g_1|_{S_{01}}$ and hence $\gamma=\tilde g_{1}(S_{01})$ is a self intersection circle of $a_{\epsilon}$ as claimed. Consider a pair of antipodal points $\pm p$ in $S_{01}$. Let $n_0$ and $n_1$ denote normal vectors of $S_{0}^{2}$ and $S_1^{2}$, respectively. Then
\[
da_{\epsilon}|_{\pm p}(n_j) = \pm d\tilde g_1|_{p}(n_j) + \epsilon db_j|_{\pm p}(n_j)\nu_j,\quad\text{for }j=0,1,
\]
since the Hopf-fiber in direction $n_j$ at $p$ is the Hopf fiber in direction $-n_j$ at $-p$. By definition, $db_j(n_j)\ne 0$ along $S_j^{2}$ and it follows that the self intersection $\gamma$ is transverse. 

By our choice of location of the fiber $S_{01}$ (far from the self intersection points of $f$) the map $\tilde g_1$ agrees with the map $\iota\circ g_1$ in a neighborhood of $S_{01}$ and in particular maps Hopf fibers $S$ near $S_{01}$ into fibers in the normal bundle of $f\colon S^{2}\to\R^{4}$. It follows that $a_{\epsilon}$ does not map any points of $S^{3}-S_{01}$ into the disk with boundary $\gamma$ in the normal bundle fiber of $f$ in which $\gamma=\iota(g_1(S_{01}))$ lies. Thus, using $D=S_1^{2}\cap b_0^{-1}(-\infty,0]$ as Seifert surface for $S_{01}\subset S^3$ and the outward normal vector field of $S_{01}\subset D$ (which is gradient like for $b_0|_{S_1^{2}}$) to shift $\gamma=a_{\epsilon}(S_{01})$ off of $a_{\epsilon}(S^{3})$ we find that the shifted curve $\gamma'$ lies in the normal fiber of $f$ outside the self intersection circle $\gamma$. 

It follows that a large annulus $A$ in fiber of $\gamma$, with one boundary component $\gamma'$, after small perturbation supported away from the boundary, intersects $f(S^{2})$ only in isolated points far from $\gamma'$. Consequently, $A$ can be made disjoint from $a_{\epsilon}(S^{3})$ by shifting in the $e_5$-direction. Capping the large annulus with a disk in a large sphere gives a disk $\bar A$ with $\pa\bar A=\gamma'$ and such that $\bar A$ is disjoint from $a_{\epsilon}(S^{3})$.  Thus, the linking number between $a_{\epsilon}(S^{3})$ and $\gamma'$ vanishes and hence $\lambda(g_1)=0$.    
\end{rem}

Melikhov studied in \cite[Example~4]{melikhov} a construction similar to the one above. More precisely, he looked at the composition $e\circ q\co S^3\to\R^4$
of the universal $8$-fold covering $q\co S^3\to Q^3$ and
an embedding $e\co Q^3\to\R^4$
of the quaternion space $Q^3=S^3/\{\pm1,\pm i,\pm j,\pm k\}$ and
showed that it represents an element with non-trivial
stable Hopf invariant in the stable $3$-stem $\pi_3^S\approx\Z_{24}$.
Here, we show that his immersion essentially coincides with $g_2$.

\begin{prop}
The immersion $e\circ q\co S^3\to\R^4$ coincides with $g_2$ up to orientation and regular homotopy.
\end{prop}

\begin{proof}
Let $\hat{e}\co{\R}P^2\hookrightarrow\R^4$ be an embedding.
Then, by \cite[\S2]{massey1}, the tubular neighborhood $N({\R}P^2)$ of $\hat{e}({\R}P^2)\subset\R^4$
is diffeomorphic to the total space of the non-orientable $D^2$-bundle over ${\R}P^2$
with Euler number $\pm2$ (see also \cite{b-j,price,p-r,yamada}).
Furthermore, 
the boundary $\partial{N}({\R}P^2)$ is known to be diffeomorphic to
the quaternion space $Q^3$ (see \cite{price}).
We thus obtain an embedding $e\co Q^3\hookrightarrow\R^4$.

For the $2$-fold covering $\rho\co S^2\to{\R}P^2$
and a suitable choice of the embedding $\hat{e}\co{\R}P^2\hookrightarrow\R^4$,
the composition $\hat{e}\circ\rho\co S^2\to\R^4$ becomes an immersion
with normal Euler class $4$.
Furthermore, if we denote by $\Gamma$ the quaternion group $\{\pm1,\pm i,\pm j,\pm k\}$ of order $8$
and put $\Z_4=\langle g\ST g^4=1\rangle$,
then from the sequence
\[\xymatrix{
\{1\}\ar[r]&\Z_4\ar[r]&\Gamma\ar[r]&\Z_2\ar[r]&\{1\},
}\]
we have the sequence of the coverings
\[\xymatrix{
S^3\ar[r]&L(4,1)=S^3/\Z_4\ar[r]&Q^3=S^3/\Gamma,
}\]
up to orientation \cite{yamada2}. Hence, we obtain the diagram
\[\xymatrix{
S^3\ar[rr]^-{\text{$4$-fold}}\ar[rrd]_-{\text{$8$-fold}}&&L(4,1)\ar[d]_-{\text{$2$-fold}}\ar[rr]_-{\text{Euler class $4$}}&&S^2\ar[d]_-{\rho}\ar[rrd]^-{F_4|_{0\text{-section}}}\\
&&Q^3\ar[rr]_-{\text{Euler number $2$}}&&{\R}P^2\ar[rr]_-{\hat{e}}&&\R^4,
}\]
in which the left part obtained by deleting the two rightmost arrows is commutative up to orientation, see \cite{lawson}, and where the rightmost triangle is commutative up to regular homotopy. The proposition follows.
\end{proof}

\section*{acknowledgments}
MT would like to thank Professor Yuichi Yamada
for illustrating various descriptions of the quaternion space and for valuable discussions, and Professor Takashi Nishimura for instructive discussions,
especially on singularities of smooth mappings. 

The authors thank an anonymous referee for careful reading of the manuscript, for correcting a sign error in Equation \eqref{eq:smale} and a mistake in the proof of Proposition \ref{prop:hirzebruch}, and for many other useful suggestions.

\end{document}